\documentclass[11pt]{amsart}
\usepackage{amscd}
\usepackage{amsmath}
\usepackage{amsxtra}
\usepackage{amsfonts}
\usepackage{amssymb}

\newtheorem{theorem}{Theorem}[section]
\newtheorem{corollary}[theorem]{Corollary}
\newtheorem{lemma}[theorem]{Lemma}
\newtheorem{proposition}[theorem]{Proposition}

\theoremstyle{definition}
\newtheorem{definition}[theorem]{Definition}
\newtheorem{remark}[theorem]{Remark}

\theoremstyle{remark}

\renewcommand{\theclaim}{\textup{\theclaim}}

\newtheorem*{acknowledgements}{Acknowledgements}

\numberwithin{equation}{section}

\def\openone

{\mathchoice

{\hbox{\upshape \small1\kern-3.3pt\normalsize1}}

{\hbox{\upshape \small1\kern-3.3pt\normalsize1}}

{\hbox{\upshape \tiny1\kern-2.3pt\SMALL1}}

{\hbox{\upshape \Tiny1\kern-2pt\tiny1}}}

\makeatletter

\newbox\ipbox

\newcommand{\ip}[2]{\left\langle #1\, , \,#2\right\rangle}
\newcommand{\diracb}[1]{\left\langle #1\mathrel{\mathchoice

{\setbox\ipbox=\hbox{$\displaystyle \left\langle\mathstrut
#1\right.$}

\vrule height\ht\ipbox width0.25pt depth\dp\ipbox}

{\setbox\ipbox=\hbox{$\textstyle \left\langle\mathstrut
#1\right.$}

\vrule height\ht\ipbox width0.25pt depth\dp\ipbox}

{\setbox\ipbox=\hbox{$\scriptstyle \left\langle\mathstrut
#1\right.$}

\vrule height\ht\ipbox width0.25pt depth\dp\ipbox}

{\setbox\ipbox=\hbox{$\scriptscriptstyle \left\langle\mathstrut
#1\right.$}

\vrule height\ht\ipbox width0.25pt depth\dp\ipbox}

}\right. }

\newcommand{\dirack}[1]{\left. \mathrel{\mathchoice

{\setbox\ipbox=\hbox{$\displaystyle \left.\mathstrut
#1\right\rangle$}

\vrule height\ht\ipbox width0.25pt depth\dp\ipbox}

{\setbox\ipbox=\hbox{$\textstyle \left.\mathstrut
#1\right\rangle$}

\vrule height\ht\ipbox width0.25pt depth\dp\ipbox}

{\setbox\ipbox=\hbox{$\scriptstyle \left.\mathstrut
#1\right\rangle$}

\vrule height\ht\ipbox width0.25pt depth\dp\ipbox}

{\setbox\ipbox=\hbox{$\scriptscriptstyle \left.\mathstrut
#1\right\rangle$}

\vrule height\ht\ipbox width0.25pt depth\dp\ipbox}

} #1\right\rangle}

\newcommand{\bz}{\mathbb{Z}}
\newcommand{\M}{\mathcal{M}}
\newcommand{\B}{\mathcal{B}}

\newcommand{\be}{\mathbb{E}}

\newcommand{\bn}{\mathbb{N}}

\def\blfootnote{\xdef\@thefnmark{}\@footnotetext}

\input xy
\xyoption{all}
\usepackage{amssymb}

\newcommand{\Span}{\overline{\operatorname*{span}}}

\def\C{\mathcal{C}}

\def\F{\mathcal{F}}

\def\be{\mathbb{E}}
\def\H{\mathcal{H}}

\def\-{^{-1}}
\def\B{\mathcal{B}}

\def\m{\mathfrak m}
\def\O{\mathcal{O}}
\def\K{\mathcal{K}}

\begin{document}

\title[Representations of Cuntz algebras associated to quasi-stationary Markov measures]{Representations of Cuntz algebras associated to quasi-stationary Markov measures}
\author{Dorin Ervin Dutkay}

\address{[Dorin Ervin Dutkay] University of Central Florida\\
	Department of Mathematics\\
	4000 Central Florida Blvd.\\
	P.O. Box 161364\\
	Orlando, FL 32816-1364\\
U.S.A.\\} \email{Dorin.Dutkay@ucf.edu}

%

\author{Palle E.T. Jorgensen}
\address{[Palle E.T. Jorgensen]University of Iowa\\
Department of Mathematics\\
14 MacLean Hall\\
Iowa City, IA 52242-1419\\}\email{palle-jorgensen@uiowa.edu}

\thanks{} 
\subjclass[2010]{47B32, 47A67, 47D07, 43A45, 42C40, 65T60, 60J27.}
\keywords{Representations in Hilbert space, wavelet representation, sigma-Hilbert space, spectral theory, harmonic analysis, absolute continuity vs singular, dichotomy, infinite product measures, Markov measures, $C^*$-algebra, Cuntz algebras, universal representation.}

\begin{abstract}
   In this paper, we answer the question of equivalence, or singularity, of two given quasi-stationary Markov measures on one-sided infinite words, and the corresponding question of equivalence of associated Cuntz algebra $\O_N$ representations. We do this by associating certain monic representations of $\O_N$ to quasi-stationary Markov measures, and then proving that equivalence for pairs of measures is decided by unitary equivalence of the corresponding pair of representations.
\end{abstract}
\maketitle \tableofcontents

\section{Introduction}

Our main theme is that of studying equivalence of pairs of measures arising in symbolic dynamics, and giving answers in terms of the representations used in generating the particular system. Departing from earlier work, we study systems derived from one-sided endomorphisms, and our measures will typically not be Gaussian.

    In our setting, we will study measures on one-sided infinite words in a finite alphabet, say realized as  $\bz_N$ (the cyclic group of order $N$), and our representations will be representations of the Cuntz algebra $\O_N$, \cite{Cu77, Cu79}. By this we mean an assignment of isometries to every letter in the fixed finite alphabet, in such a way that the distinct $N$ isometries in a Hilbert space $\H$, have orthogonal ranges, adding up to the identity operator in $\H$. This particular family of representations is motivated in part by quantization: since there is no symmetry or anti-symmetry restriction on occupancy of states, this representation suggests Boltzmann statistics.  Now the ``quantized'' system must be realized as an $L^2$-space with respect to a suitable measure on $\K_N$, the Cantor group of infinite words on $N$ letters.  Because of the setting, this will typically not be possible with the use of more traditional Gaussian measures (on infinite product spaces); and as an alternative we suggest a family of quasi-stationary Markov measures. Now the symbolic representation of $\K_N$ suggests a one-sided shift to the left, and a system of $N$ inverse branches of shift to the right,  by filling in a letter from $\bz_N$. This fact, in turn, suggests the use of Markov measures, see e.g., \cite{Ak12, JoPa12}.
    
      As  $C^*$-algebras (denoted $\O_N$), the Cuntz algebras are indexed by an integer $N > 1$, where $N$ is the number of generators, see Definition \ref{def2.0}; but rather than the $C^*$-algebras themselves, it is their representations that offer surprises. Indeed, by analogy to the theory of representations of infinite discrete groups,  for $\O_N$, it is she set of equivalence classes of its representations that offer surprising insight into such diverse areas as harmonic analysis, and also into a host of problems from pure and applied mathematics: wavelets, signal processing, ergodic theory, mathematical physics, and more; see e.g., \cite{DuJo06a, DuJo12, DuJo11a, DuJo07a, Jor06, BrJo02}.

   Now $\O_N$ is a simple, purely infinite $C^*$-algebra, \cite{Cu77}, and its $K$-groups are known. But the question of ``finding'' its irreducible representations is a highly subtle one. In fact, it is known \cite{Gli60} that the equivalence classes of representations of $\O_N$,  for fixed $N$,  does not even lend itself to a Borel cross section; more precisely, the set of  equivalence classes, under unitary equivalence, does not admit a parameterization in the measurable Borel category, hence does not admit a Borel cross-section. Intuitively, the representations defy classification.

     Nonetheless, special families of inequivalent representations have been found up to now, and they have applications in mathematical physics \cite{BrJo02, Bur04, GN07, AK08, Kaw03, Kaw06, Kaw09, KHL09}, in wavelets \cite{DuJo08a, DuJo07a, Jor06, Jor01}, harmonic analysis \cite{Str89, DuJo9a, DuJo07b}, and in fractal geometry \cite{DuJo06a, DuJo11a}.  Hence it is of interest to identify both discrete and continuous series of representations of $\O_N$. The particular representations considered here (Definition \ref{def2.7}) turn out to be directly associated with families of Markov measures (and therefore Markov processes); and the correspondence in both directions yields new insights, see sections 3 and 4 below. For the representations, the question is irreducibility, and for the  Markov measures, it is ergodicity and ``properties at infinity''.

      When associated representations of the Cuntz algebra are brought to bear, we arrive at useful non-commutative versions of these (commutative) symbolic shift mappings. For earlier related work, see e.g., \cite{DPS14, DuJo12, Sk970, GPS95}.

      Now by comparison, in the more traditional instances of the dynamics problem in its commutative and non-commutative guise \cite{BrRo81, Hi80}, there are three typical methods of attacking the question of equivalence or singularity of two measures: (1) Kakutani's theorem for infinite-product measures \cite{Ka48}, which asserts that two infinite product measures are either equivalent or mutually singular, (2)  methods based  on entropy considerations, and (3) the method of using the theory of reproducing kernels; see \cite{Hi80}. Because of the nature of our setting, one-sided shifts, commutative and non-commutative, we must depart from the setting of Gaussian measures. As a result, of these three traditional approaches to the question of equivalence or singularity of two measures, vs equivalence of representations, only ideas from (1) seem to be applicable to our present setting.
In the present paper, we develop this and we answer the question for quasi-stationary Markov measure in Corollary \ref{cor2.14}, and Theorem \ref{th2.15}. In Theorem \ref{th3.7} we show that the quasi-stationary Markov measures are ergodic, and that the associated  $\O_N$ representations are irreducible.

      Recently there has been an increased interest in use of the Cuntz algebras and their representations in dynamics (including the study of fractals, and geometric measure theory), in ergodic theory, and in quantization questions from physics. Perhaps this is not surprising since Cuntz algebras are infinite algebras on a finite number of generators, and defined from certain relations; in this case, the Cuntz relations. By their nature, these representations reflect intrinsic selfsimilar inherent in the problem at hand; and thus they serve ideally to encode iterated function systems (IFSs), their dynamics, and their measures. At the same time, the $\O_N$-representations offer (in a more subtle way) a new harmonic analysis of IFS-fractal measures. Even though the Cuntz algebras initially entered into the study of operator-algebras and physics, in recent years these same Cuntz algebras, and their representation, have found increasing use in pure and applied problems, such as wavelets, fractals, and signals.
      
      The paper is structured as follows: in section 2 we recall some definitions and facts about Cuntz algebras and their monic representations. In section 3 we define quasi-stationary Markov measures (Definition \ref{def2.1}), present some statements equivalent to the quasi-stationary condition and some consequences (Proposition \ref{pr2.13}, Proposition \ref{pr3.3}), show that such Markov measures give rise to monic representations of the Cuntz algebra (Theorem \ref{tha2.3}), and show that these measures are ergodic for the shift and the $\O_N$-representations are irreducible (Theorem \ref{th3.7}). In section 4, we study when two quasi-stationary Markov measures are equivalent or mutually singular, and the similar question for the associated representations of the Cuntz algebra. We present a dichotomy theorem (Theorem \ref{th4.1}) which shows that these are the only two possibilities, in the spirit of Kakutani's work \cite{Ka48}. We present some necessary and sufficient conditions for equivalence and similarly for singularity (Proposition \ref{pr2.12}). In Theorem \ref{th2.15}, we show that, under slightly stronger assumptions, the Markov measure is equivalent to a stationary one and so are the associated representations of $\O_N$.

\section{Preliminaries: The Cuntz algebra and symbolic dynamics}
 We set the stage for our correspondence between measures on the Cantor group $\K_N$ on the one side, and representations of the $C^*$-algebra $\O_N$, on the other. Our measures here are prescribed by a fixed system of Markov transition matrices; we call them ``quasi-stationary Markov measures''.

\begin{definition}\label{def2.0}
Let $N\geq 2$. The Cuntz algebra $\O_N$ is the $C^*$-algebra generated by a system of $N$ isometries $(S_i)_{i\in\bz_N}$ satisfying the {\it Cuntz relations}
\begin{equation}
S_i^*S_j=\delta_{ij}I,\quad (i,j\in\bz_N),\quad \sum_{i\in\bz_N}S_iS_i^*=I.
\label{eq2.0.1}
\end{equation}
\end{definition}
\begin{definition}\label{def0.1}
Fix an integer $N\geq 2$. Let $\bz_N:=\{0,1,\dots,N-1\}$. Let $(S_i)_{i\in\bz_N}$ be a representation of the Cuntz algebra $\O_N$ on a Hilbert space $\H$. We will call elements in $\bz_N^k$ {\it words of length $k$}. 
We denote by $\K=\K_N=\bz_N^{\bn}$, the set of all infinite words.
Given two finite words $\alpha=\alpha_1\dots\alpha_n$, $ \beta=\beta_1\dots\beta_m$, we denote by $\alpha\beta$ the concatenation of the two words, so $\alpha\beta=\alpha_1\dots\alpha_n\beta_1\dots\beta_m$. Similarly, for the case when $\beta$ is infinite. Given a word $\omega=\omega_1\omega_2\dots$, and $k$ a non-negative integer smaller than its length, we denote by 
$$\omega|k:=\omega_1\dots\omega_k,$$
the truncated word.

For a finite word $I=i_1\dots i_n$, we denote by 
$$S_I:=S_{i_1}\dots S_{i_n}.$$

We define $\mathfrak A_N$ to be the abelian subalgebra of $\O_N$ generated by $S_IS_I^*$, for all finite words $I$. 
As a $C^*$-algebra, $\mathfrak A_N$ is naturally isomorphic to $C(\K_N)$, the continuous functions on the Cantor group $\K_N$, see Definition \ref{defcantor}.

We say that a subspace $M$ is {\it $S_i^*$-invariant} if $S_i^*M\subset M$ for all $i\in\bz_N$. Equivalently $$P_MS_i^*P_M=S_i^*P_M,$$ where $P_M$ is the projection onto $M$. We say that $M$ is {\it cyclic } for the representation if 
$$\Span\{S_IS_J^*v : v\in M, I,J\mbox{ finite words }\}=\H.$$

\end{definition}

\begin{definition}\label{defcantor}
Fix an integer $N\geq 2$.  The {\it Cantor group on $N$ letters} is 
$$\K=\K_N=\prod_{1}^\infty\bz_N=\{(\omega_1\omega_2\dots) : \omega_i\in \bz_N\mbox{ for all }i=1,2,\dots\},$$
an infinite Cartesian product.

The elements of $\K_N$ are infinite words.
On the Cantor group, we consider the product topology. We denote by $\B(\K_N)$ the sigma-algebra of Borel subsets of $\K_N$. We denote by $\M(\K_N)$ the set of all finite Borel measures on $\K_N$.

Denote by $\sigma$ the shift on $\K_N$, $\sigma(\omega_1\omega_2\dots)=\omega_2\omega_3\dots$. Define the inverse branches of $\sigma$: for $i\in\bz_N$, $\sigma_i(\omega_1\omega_2\dots)=i\omega_1\omega_2\dots$.

For a finite word $I=i_1\dots i_k\in \bz_N^k$, we define the corresponding {\it cylinder set }

\begin{equation}
\C(I)=\{\omega\in \K_N : \omega_1=i_1,\dots,\omega_k=i_k\}=\sigma_{i_1}\dots\sigma_{i_n}(\K_N).
\label{eqcantor0}
\end{equation}
\end{definition}

\begin{definition}\label{def2.3}
Let $(S_i)_{i\in\bz_N}$ be a representation of the Cuntz algebra $\O_N$ on a Hilbert space $\H$. Then we define the projection valued function $P$ on the Borel sigma-algebra $\B(\K_N)$ by defining it on cylinders first
\begin{equation}
P(\C(I))=S_IS_I^*\mbox{ for any finite word }I.
\label{eq2.3.1}
\end{equation}
and then extending it by the usual Kolmogorov procedure (see \cite{DHJ13} for details). We call this, the {\it projection valued measure associated to the representation}. 
This projection valued measure then induces a representation $\pi$ of bounded (and of continuous) functions on $\K_N$, by setting
\begin{equation}
\pi(f)=\int_{\K_N} f(\omega)\,dP(\omega).
\label{eq2.3.2}
\end{equation}

For every $x\in \H$, define the Borel measure $\m_x$ on $\K_N$ by
\begin{equation}
\m_x(A)=\ip{x}{P(A)x},\quad(A\in\B(\K_N)).
\label{eq2.3.3}
\end{equation}
Using \eqref{eq2.3.2} and \eqref{eq2.3.3}, and the property
$$P(A\cap B)=P(A)P(B),\quad(A,B\in\B(\K_N)),$$
one obtains
\begin{equation}
\left\|\pi(f)x\right\|_\H^2=\int|f|^2\,d\m_x.
\label{eq2.3.4}
\end{equation}

We will also use the notations
$$P(I)=P(\C(I))\mbox{ for }I=i_1\dots i_n\mbox{ and $P(\omega)=P(\{\omega\})$ for $\omega\in\K_N$}.$$
\end{definition}

\begin{definition}\label{def2.4}
We say that a representation of the Cuntz algebra $\O_N$ on a Hilbert space $\H$ is {\it monic} if there is a cyclic vector $\varphi$ in $\H$ for the abelian subalgebra $\mathfrak A_N$, i.e., 
$$\Span\{S_IS_I^*\varphi : I\mbox{ finite  word }\}=\H.$$
\end{definition}

\begin{definition}\label{def2.7}
A {\it monic system} is a pair $(\mu,(f_i)_{i\in\bz_N})$ where $\mu$ is a finite Borel measure on $\K_N$ and $f_i$ are some functions on $\K_N$ such that $\mu\circ\sigma_i^{-1}\ll\mu$ for all $i\in\bz_N$ and 
\begin{equation}
\frac{d(\mu\circ\sigma_i^{-1})}{d\mu}=|f_i|^2,
\label{eq2.5.1}
\end{equation}
for some functions $f_i\in L^2(\mu)$ with the property that 
\begin{equation}
f_i(x)\neq 0\mbox{ for $\mu$-a.e. $x$ in $\sigma_i(\K_N)$}.
\label{eq2.5.2}
\end{equation}

We say that a monic system is {\it nonnegative} if $f_i\geq0$ for all $i\in\bz_N$.

The representation of $\O_N$ associated to a monic system is 
\begin{equation}
S_if=f_i(f\circ\sigma),\quad(i\in\bz_N,f\in L^2(\mu)),
\label{eq2.5.3} 
\end{equation}
where $\sigma$ is the one-sided shift, see Definition \ref{defcantor} and \cite[Theorem 2.7]{DJ14}.

We say that this representation $(S_i)_{i\in\bz_N}$ of the Cuntz algebra is nonnegative if the monic system is. 
\end{definition}

\begin{theorem}\label{th2.5}\cite{DJ14}
Let $(S_i)_{i\in\bz_N}$ be a representation of $\O_N$. The representation is monic if and only if it is unitarily equivalent to a representation associated to a monic system.
\end{theorem}

\section{Representations of the Cuntz algebra and quasi-stationary Markov measures}

 Here we define quasi-stationary Markov measures on $\K_N$ and we associate to them monic representations of the Cuntz algebra $\O_N$. We show (Theorem \ref{th3.7}) that the quasi-stationary Markov measures are ergodic, and that the associated $\O_N$ representations are irreducible.

\begin{definition}\label{def2.1} 
An $N\times N$ matrix $T$ is called {\it stochastic } if the sum of the entries in each row is 1, i.e.,
\begin{equation}
Te=e\mbox{ for all $n\in\bn$, where }e= \begin{pmatrix}
	1\\
	1\\
	\vdots\\
	1
\end{pmatrix},
\label{eq2.1.2}
\end{equation}
(We use column vectors for multiplication on the right and row vectors for multiplication on the left).

Let $\lambda=(\lambda_0,\dots,\lambda_{N-1})$ be a {\it positive probability vector}, i,e., a vector of positive numbers such that $\sum_{i\in\bz_N}\lambda_i=1$ and let $\{T^{(n)}\}_{n\in\bn}$ be a sequence of stochastic $N\times N$ matrices with positive entries such that
\begin{equation}
\lambda T^{(1)}=\lambda
\label{eq2.1.1}
\end{equation}
\begin{equation}
\lambda_i>0, T_{i,j}^{(n)}>0,\quad(i,j\in\bz_N,n\in\bn).
\label{eq2.1.3}
\end{equation}
Define the Borel probability measure $\mu=\mu_{\lambda,T}$ on $\K_N$, first, on cylinder sets: for $I=i_1\dots i_n$
\begin{equation}
\mu(\C(I))=\lambda_{i_1}T_{i_1,i_2}^{(1)}T_{i_2,i_3}^{(2)}\dots T_{i_{n-1},i_n}^{(n-1)},
\label{eq2.1.4}
\end{equation}
and then extending it to the Borel sigma-algebra, using the Kolmogorov extension theorem (see Remark \ref{rem2.2} below). We say that $\mu=\mu_{\lambda,T}$ is the {\it Markov measure associated to} $\lambda$ and $T$.

We say that the Markov measure $\mu$ is {\it quasi-stationary} if 
\begin{equation}
\sum_{n=1}^\infty \left|\frac{T_{x_n,x_{n+1}}^{(n)}}{T_{x_n,x_{n+1}}^{(n+1)}}-1\right|<\infty,\mbox{ for all }x_1x_2\dots\in\K_N.
\label{eq2.1.5}
\end{equation}
\end{definition}

\begin{remark}\label{rem2.2}
We check the Kolmogorov consistency conditions: take $I=i_1\dots i_{n}$. Then $\C(i_1\dots i_{n})$ is the disjoint union of $\C(i_1\dots i_nj)$, $j\in\bz_N$. We have 
$$\sum_{j\in\bz_N}\mu(\C(i_1\dots i_nj))=\sum_{j\in\bz_N}\lambda_{i_1}T_{i_1,i_2}^{(1)}\dots T_{i_{n-1},i_n}^{(n-1)}T_{i_n,j}^{(n)}=\lambda_{i_1}T_{i_1,i_2}^{(1)}\dots T_{i_{n-1},i_n}^{(n-1)}=\mu(\C(i_1\dots i_n)).$$

\end{remark}

Here are some conditions, equivalent to the quasi-stationary condition, which are easier to check:
\begin{proposition}\label{pr2.13}
Let $T^{(n)}$, $n\in\bz$, be a sequence of stochastic matrices with positive entries. The following statements are equivalent
\begin{enumerate}
	\item For all $x_1x_2\dots$ in $\K_N$
	$$\sum_{n=1}^\infty \left|\frac{T_{x_n,x_{n+1}}^{(n)}}{T_{x_n,x_{n+1}}^{(n+1)}}-1\right|<\infty.$$
	\item For all $i,j\in\bz_N$,
	\begin{equation}
\sum_{n=1}^\infty \left|\frac{T_{i,j}^{(n)}}{T_{i,j}^{(n+1)}}-1\right|<\infty.
\label{eq2.13.1}
\end{equation}
\item For all $i,j\in\bz_N$, 
\begin{equation}
\sum_{n=1}^\infty|T_{i,j}^{(n)}-T_{i,j}^{(n+1)}|<\infty,
\label{eq2.13.2}
\end{equation}
and
\begin{equation}
\inf_n T_{i,j}^{(n)}>\infty.
\label{eq2.13.3}
\end{equation}
\end{enumerate} 
Under these conditions, the matrices $T^{(n)}$ converge to a stochastic matrix $T^{(\infty)}$ with positive entries. 
\end{proposition}

\begin{proof}
Assume (i). Let $x_{2n}=i$, $x_{2n+1}=j$ for all $n\in\bn$. Then
$$\sum_n \left|\frac{T_{i,j}^{(2n)}}{T_{i,j}^{(2n+1)}}-1\right|+\left|\frac{T_{j,i}^{(2n+1)}}{T_{j,i}^{(2n+2)}}-1\right|<\infty.$$
Switching between $i$ and $j$ we obtain
$$\sum_n \left|\frac{T_{j,i}^{(2n)}}{T_{j,i}^{(2n+1)}}-1\right|+\left|\frac{T_{i,j}^{(2n+1)}}{T_{i,j}^{(2n+2)}}-1\right|<\infty.$$
Then
$$\sum_n \left|\frac{T_{i,j}^{(n)}}{T_{i,j}^{(n+1)}}-1\right|=\sum_n \left|\frac{T_{i,j}^{(2n)}}{T_{i,j}^{(2n+1)}}-1\right|+\left|\frac{T_{i,j}^{(2n+1)}}{T_{i,j}^{(2n+2)}}-1\right|<\infty.$$

Conversely, assume (ii) and take $x_1x_2\dots\in\K_N$. Then 
$$\sum_n \left|\frac{T_{x_n,x_{n+1}}^{(n)}}{T_{x_n,x_{n+1}}^{(n+1)}}-1\right|=\sum_{(i,j)\in\bz_N\times\bz_N}\sum_{n: (x_n,x_{n+1})=(i,j)}\left|\frac{T_{i,j}^{(n)}}{T_{i,j}^{(n+1)}}-1\right|<\infty.$$
This proves (i).

Condition (ii) implies that the infinite product 
$$A_{i,j}:=\prod_{n=1}^\infty \frac{T_{i,j}^{(n)}}{T_{i,j}^{(n+1)}}$$
is convergent to a positive number, for all $i,j\in\bz_N$. 
But $$T_{i,j}^{(n)}=T_{i,j}^{(1)}\frac{1}{\prod_{k=1}^{n-1} \frac{T_{i,j}^{(k)}}{T_{i,j}^{(k+1)}}}\rightarrow \frac{T_{i,j}^{(1)}}{A_{i,j}}=:T_{i,j}^{(\infty)}>0.$$
Since each matrix $T^{(n)}$ is stochastic, it follows that also the matrix $T^{(\infty)}$ is stochastic.

Assume now (ii) and we prove (iii). Since $T_{i,j}^{(n)}$ converges to a positive number, \eqref{eq2.13.3} follows. Then, since $T_{i,j}^{(n)}\leq 1$, we have
$$\sum_{n=1}^\infty|T_{i,j}^{(n)}-T_{i,j}^{(n+1)}|\leq \sum_{n=1}^\infty \left|\frac{T_{i,j}^{(n)}}{T_{i,j}^{(n+1)}}-1\right|<\infty.$$

Assume (iii) and we prove (ii).  Let $c$ be the infimum in \eqref{eq2.13.3}. We have
$$\sum_{n=1}^\infty \left|\frac{T_{i,j}^{(n)}}{T_{i,j}^{(n+1)}}-1\right|\leq \frac1c\sum_{n=1}^\infty|T_{i,j}^{(n)}-T_{i,j}^{(n+1)}|<\infty.$$
\end{proof}

\begin{proposition}\label{pr3.3}
Let $T^{(n)}$ be a sequence of stochastic matrices with positive entries, satisfying the quasi-stationary condition \eqref{eq2.1.5}. Let $T^{(\infty)}$ be the limit of the sequence $T^{(n)}$ as in Proposition \ref{pr2.13}. Let $v=(v_0,\dots,v_{N-1})$ be the Perron-Frobenius positive probability eigenvector for $T^{(\infty)}$, i.e., $vT^{(\infty)}=v$, $v_i>0$, $i\in\bz_N$ and $\sum_{i\in\bz_N}v_i=1$. Let $Q$ be the $N\times N$ matrix with identical rows equal to $v$. Then, for all $p\geq 1$, 
$$\lim_{n\rightarrow\infty}\prod_{k=p+1}^nT^{(k)}=\lim_{n\rightarrow\infty}T^{(p+1)}T^{(p+2)}\dots T^{(n)}=Q.$$

\end{proposition}

\begin{proof}
The result follows from \cite[Theorem 4.14, page 150]{Sen06}: the matrices $T^{(n)}$ and $T^{(\infty)}$ have positive entries, so $T^{(\infty)}$ is regular as required. 

\end{proof}

\begin{theorem}\label{tha2.3}
Let $\mu$ be a quasi-stationary Markov measure. Then $\mu\circ\sigma_j^{-1}\ll\mu$ for all $j\in\bz_N$ and
\begin{equation}
\frac{d(\mu\circ\sigma_j^{-1})}{d\mu}(x_1x_2\dots)=\delta_{j,x_1}\frac{\lambda_{x_2}}{\lambda_{x_1}T_{x_1,x_2}^{(1)}}F(x_2x_3\dots)=:f_j^2(x_1x_2\dots),\, f_j\geq 0,
\label{eqa2.3.2}
\end{equation}
where
\begin{equation}
F(x_1x_2\dots)=\prod_{n=1}^\infty\frac{T_{x_n,x_{n+1}}^{(n)}}{T_{x_n,x_{n+1}}^{(n+1)}}\mbox{ and }0<F(x_1x_2\dots)<\infty.
\label{eqa2.3.3}
\end{equation}
Also,
\begin{equation}
\frac{d(\mu\circ\sigma^{-1})}{d\mu}(x_1x_2\dots)=\frac{1}{F(x_1x_2\dots)}.
\label{eqa2.3.4}
\end{equation}

The operators $S_j$ on $L^2(\mu)$ defined by 
\begin{equation}
S_jf=f_j(f\circ\sigma),\quad(f\in L^2(\mu),j\in\bz_N)
\label{eqa2.3.4.0}
\end{equation}
form a nonnegative monic representation of the Cuntz algebra $\O_N$. 
\end{theorem}

\begin{proof}
We use \cite[Theorem 5]{EnSh80}. Consider, for $n\in\bn$, the sigma-algebras $\F_n$ generated by cylinder sets $\C(I)$ with $I$ of length $n$. Let $P_n$ be the restriction of $\mu$ to $\F_n$ and let $Q_n$ be the restriction of $\mu\circ\sigma_j^{-1}$ to $\F_n$. We claim that, for $n\geq 3$, $Q_n\ll P_n$ and
\begin{equation}
\frac{dQ_n}{dP_n}(x_1x_2\dots)=\delta_{j,x_1}\frac{\lambda_{x_2}}{\lambda_{x_1}T_{x_1,x_2}^{(1)}}\prod_{k=1}^{n-2}\frac{T_{x_{k+1},x_{k+2}}^{(k)}}{T_{x_{k+1},x_{k+2}}^{(k+1)}}=:Z_n(x_1x_2\dots x_n)=:Z_n(x_1x_2\dots)
\label{eqa2.3.1}
\end{equation}
Indeed, for $I=i_1\dots i_n$ we have
$$\mu(\sigma_j^{-1}(\C(I)))=\mu(\{x_1x_2\dots : j=i_1,x_1=i_2,\dots, x_{n-1}=i_n\})=\delta_{j,i_1}\lambda_{i_2}T_{i_2,i_3}^{(1)}\dots T_{i_{n-1},i_n}^{(n-2)}$$$$=Z_n(i_1\dots i_n)\cdot\lambda_{i_1}T_{i_1,i_2}^{(1)}T_{i_2,i_3}^{(2)}\dots T_{i_{n-1},i_n}^{(n-1)}=Z_n(i_1\dots i_n)\mu(\C(I)).$$
This implies \eqref{eqa2.3.1}.

Note that \eqref{eq2.1.5} implies that the infinite product in \eqref{eqa2.3.3} is convergent and that $0<F(x)<\infty$, for all $x\in\K_N$. Therefore 
$$Z_\infty(x_1x_2\dots):=\lim_n Z_n(x_1x_2\dots)=\delta_{j,x_1}\frac{\lambda_{x_2}}{\lambda_{x_1}T_{x_1,x_2}^{(1)}}F(x_2x_3\dots).$$
Since $Z_\infty(x_1x_2\dots)<\infty$ for all $x_1x_2\dots\in\K_N$, \cite[Theorems 4,5]{EnSh80} implies that $\mu\circ\sigma_j^{-1}\ll\mu$ and $\frac{d(\mu\circ\sigma_j^{-1})}{d\mu}=Z_\infty$, which is \eqref{eqa2.3.3}.

To check \eqref{eqa2.3.4}, we can proceed in the same way; or we can use the next lemma:

\begin{lemma}\label{lem2.10}\cite[Proposition 2.8]{DJ14}
Let $(\mu,(f_i)_{i\in\bz_N})$ be a monic system. Then $\mu\circ\sigma^{-1}\ll\mu$, and
\begin{equation}
\sum_{j\in\bz_N}\frac{1}{|f_j\circ\sigma_j|^2}=\frac{d(\mu\circ\sigma^{-1})}{d\mu}.
\label{eq1.8.1}
\end{equation}
\end{lemma}

In our context, we obtain that:
$$\frac{d(\mu\circ\sigma^{-1})}{d\mu}(x_1x_2\dots)=\sum_{j\in\bz_N}\frac{1}{\frac{d(\mu\circ\sigma_j^{-1})}{d\mu}\circ\sigma_j(x_1x_2\dots)}=\sum_{j\in\bz_N}\frac{\lambda_jT_{j,x_1}^{(1)}}{\lambda_{x_1}}\frac{1}{F(x_1x_2\dots)}$$$$=\frac{\lambda_{x_1}}{\lambda_{x_1}F(x_1x_2\dots)}=\frac{1}{F(x_1x_2\dots).}$$

Theorem \ref{th2.5} implies that the operators $S_j$ define a nonnegative monic representation of $\O_N$. 
\end{proof}

\begin{definition}\label{def2.11}
Let $\mu$ be a quasi-stationary Markov measure associated to $(\lambda,T)$. Let
\begin{equation}
f_j(x_1x_2\dots)=\delta_{j,x_1}\sqrt{\frac{\lambda_{x_2}}{\lambda_{x_1}T_{x_1,x_2}^{(1)}}F(x_2x_3\dots)},\quad(x_1x_2\dots\in\K_N)
\label{eq2.11.1}
\end{equation}
where 
\begin{equation}
F(x_1x_2\dots)=\prod_{n=1}^\infty\frac{T_{x_n,x_{n+1}}^{(n)}}{T_{x_n,x_{n+1}}^{(n+1)}}.
\label{eqa2.3.3a}
\end{equation}
We call the representation of the Cuntz algebra $\O_N$ on $L^2(\mu)$ defined by
\begin{equation}
S_jf=f_j(f\circ\sigma),
\label{eq2.11.2}
\end{equation}
{\it the representation of $\O_N$ associated to $(\lambda,T)$ (or associated to the measure $\mu$)}.
\end{definition}

\begin{theorem}\label{th3.7}
Let $\mu$ be a quasi-stationary Markov measure associated to $(\lambda,T)$ and let $(S_i)_{i\in\bz_N}$ the corresponding representation of $\O_N$. Then
\begin{enumerate}
	\item The measure $\mu$ is ergodic with respect to the endomorphism $\sigma$. 
	\item The representation of $\O_N$ is irreducible. 
\end{enumerate}

\end{theorem}

\begin{proof}
Let $T^{(\infty)}$ be the limit of the matrices $T^{(n)}$ as in Proposition \ref{pr2.13}. We know that $T^{(\infty)}$ is a stochastic matrix with positive entries. Let $v$ be the Perron-Frobenius positive probability left-eigenvector for $T^{(\infty)}$ as in Proposition \ref{pr3.3}. Let $\mu_\infty$ be the stationary Markov measure associated to $v$ and the constant sequence $T^{(\infty)}$.

We will prove the following
\begin{lemma}\label{lem3.8}
For any Borel sets $E$ and $F$ in $\K_N$
\begin{equation}
\lim_{k\rightarrow\infty}\mu(\sigma^{-k}(E)\cap F)=\mu_\infty(E)\mu(F).
\label{eq3.7.1}
\end{equation}
\end{lemma}
\begin{proof}
Since the unions of cylinder sets form a monotone class that generates the Borel sigma algebra on $\K_N$, it is enough to check \eqref{eq3.7.1} on cylinder sets. Take two finite words $I=i_1\dots i_l$ and $J=j_1\dots j_r$. For $k$ large, we have:
$$\mu(\sigma^{-k}(\C(I))\cap \C(J))=\mu(\{x_1x_2\dots\in\K_N : x_1=j_1,\dots, x_r=j_r,x_{k+1}=i_1,\dots, x_{k+l}=i_l\})$$
$$=\sum_{x_{r+1},\dots,x_k\in\bz_N}\mu(\C(j_1\dots j_rx_{r+1}\dots x_ki_1\dots i_l))$$
$$=\sum_{x_{r+1},\dots, x_k}\lambda_{j_1}T_{j_1,j_2}^{(1)}\dots T_{j_{r-1},j_r}^{(r-1)}T_{j_r,x_{r+1}}^{(r)}T_{x_{r+1},x_{r+2}}^{(r+1)}\dots T_{x_{k-1},x_k}^{(k-1)}T_{x_k,i_1}^{(k)}T_{i_1,i_2}^{(k+1)}\dots T_{i_{l-1},i_l}^{(k+l-1)}$$
$$=\mu(\C(J))(T^{(r)}T^{(r+1)}\dots  T^{(k)})_{j_r,i_1}T_{i_1,i_2}^{(k+1)}\dots T_{i_{l-1},i_l}^{(k+l-1)}\mbox{ and, by Proposition \ref{pr3.3}.}$$$$\stackrel{k\rightarrow\infty}{\rightarrow}\mu(\C(J))v_{i_1}T_{i_1,i_2}^{(\infty)}\dots T_{i_{l-1},i_l}^{(\infty)}=\mu(\C(J))\mu_\infty(\C(I)).$$
\end{proof}

Now take a Borel set $A$ with $\sigma^{-1}(A)=A$ and apply \eqref{eq3.7.1} with $E=F=A$: we have
$$\mu_\infty(A)\mu(A)=\lim_{k\rightarrow\infty}\mu(\sigma^{-k}(A)\cap A)=\mu(A).$$
This means that $\mu(A)=0$, or $\mu_\infty(A)=1$. The same argument can be applied to $\K_N\setminus A$ to obtain $\mu(\K_N\setminus A)=0$, or $\mu_\infty(\K_N\setminus A)=1$. Therefore, if $\mu(A)\neq 0$ then $\mu_\infty(A)=1$, so $\mu_\infty(\K_N\setminus A)=0$, which implies that $\mu(\K_N\setminus A)=0$. 

This shows that $\mu$ is ergodic. 

The statement in (ii) follows from (i) by \cite[Theorem 2.13]{DJ14}.
\end{proof}

\section{Equivalence of measures vs equivalence of representations}
Here we study equivalence and orthogonality of two quasi-stationary Markov measures and the similar problem for the associated representations of the Cuntz algebra $\O_N$. We begin with a dichotomy theorem that reminds us of, and has overlap with, the work of Kakutani \cite{Ka48}.

\begin{theorem}\label{th4.1}
Let $\mu$ and $\mu'$ be two quasi-stationary Markov measures. Then $\mu$ and $\mu'$ are either equivalent, or mutually singular. The corresponding representations of $\O_N$ are either equivalent or disjoint. 
\end{theorem}

\begin{proof}
By Theorem \ref{th3.7}, the two associated monic representations of $\O_N$ are irreducible. Therefore they are either equivalent or disjoint. By \cite[Proposition 2.10 and Theorem 2.12]{DJ14}, the two measures are therefore either equivalent or mutually singular. 
\end{proof}

\begin{proposition}\label{pr2.12}
Let $\mu$ and $\mu'$ be two quasi-stationary Markov measures associated to $(\lambda,T)$ and $(\lambda',T')$ respectively. Let $\pi(\O_N)=(S_i)_{i\in\bz_N}$ and $\pi'(\O_N)=(S_i')_{i\in\bz_N}$ be the two associated representations of the Cuntz algebra $\O_N$. Define $Z_\infty$ on $\K_N$ by
\begin{equation}
Z_\infty(x_1x_2\dots)=\frac{\lambda_{x_1}'}{\lambda_{x_1}}\limsup_n\prod_{k=1}^{n-1}\frac{T_{x_k,x_{k+1}}'^{(k)}}{T_{x_k,x_{k+1}}^{(k)}}
\label{eq2.12.1}
\end{equation}

The following statements are equivalent:
\begin{enumerate}
	\item The representations $\pi(\O_N)$ and $\pi'(\O_N)$ are equivalent.
	\item The measures $\mu$ and $\mu'$ are equivalent. 
	\item $\mu'(Z_\infty<\infty)=\mu(Z_\infty>0)=1$.
	\item For $\mu$- and $\mu'$- almost every $i_1i_2\dots \in\K_N$,
	\begin{equation}
\sum_{n=1}^\infty\left(1-\sum_{x_{n+1}\in\bz_N}\sqrt{T_{i_n,x_{n+1}}'^{(n)}}\sqrt{T_{i_n,x_{n+1}}^{(n)}}\right)<\infty.
\label{eq2.12.2}
\end{equation}
\end{enumerate}

Also, the following statements are equivalent:
\begin{enumerate}
	\item[(a)] The representations $\pi(\O_N)$ and $\pi'(\O_N)$ are disjoint.
	\item[(b)] The measures $\mu$ and $\mu'$ are mutually singular. 
	\item[(c)] $\mu'(Z_\infty=\infty)=1$ or $\mu(Z_\infty=0)=1$.
	\item[(d)] For $\mu'$-a.e. $i_1i_2\dots\in\K_N$,
	\begin{equation}
\sum_{n=1}^\infty\left(1-\sum_{x_{n+1}\in\bz_N}\sqrt{T_{i_n,x_{n+1}}'^{(n)}}\sqrt{T_{i_n,x_{n+1}}^{(n)}}\right)=\infty.
\label{eq2.12.3}
\end{equation}
\end{enumerate}
\end{proposition}

\begin{proof}
The equivalence of (i) with (ii) and of (a) with (b) follow from \cite[Proposition 2.10,Theorem 2.12]{DJ14}. For the equivalence of (ii) with (iii) and of (b) with (c) we use \cite[Corollary 6]{EnSh80}.
We have the sigma algebras $\F_n$ generated by cylinder sets $\C(I)$ with $I$ of length $n$ and we let $P_n,Q_n$ be the restrictions of $\mu,\mu'$ to $\F_n$. Then the measures $P_n$ and $Q_n$ are equivalent and 
$$Z_n(x_1x_2\dots)=\frac{dQ_n}{dP_n}(x_1x_2\dots)=\frac{\lambda'_{x_1}}{\lambda_{x_1}}\prod_{k=1}^{n-1}\frac{T_{x_k,x_{k+1}}'^{(k)}}{T_{x_k,x_{k+1}}^{(k)}}.$$
Then $Z_\infty=\limsup Z_n$. Everything now follows from \cite[Corollary 6]{EnSh80}.

For the equivalence (ii)$\Leftrightarrow$(iv) and (b)$\Leftrightarrow$(d), we use \cite[Theorem 4, page 528]{Shi80}, which asserts, in our context, that $\mu'\ll\mu$ if and only if 
\begin{equation}
\sum_{n=1}^\infty\left(1-\be\left(\sqrt{\frac{Z_{n+1}}{Z_n}} \,|\,\F_{n}\right)\right)<\infty
\label{eq2.12.4}
\end{equation}
$\mu'$-a.e., and $\mu\perp\mu'$ if the sum in \eqref{eq2.12.4} is infinite $\mu'$-a.e.
The equivalences then follow directly from this because 
$$\be\left(\sqrt{\frac{Z_{n+1}}{Z_n}} \,|\,\F_{n}\right)(i_1\dots i_n)=\sum_{x_{n+1}\in\bz_N}\sqrt{\frac{T_{i_n,x_{n+1}}'^{(n)}}{T_{i_n,x_{n+1}}^{(n)}}}\cdot T_{i_n,x_{n+1}}^{(n)}.$$
\end{proof}

\begin{corollary}\label{cor2.14}
Let $\mu$ and $\mu'$ be two quasi-stationary Markov measures associated to $(\lambda,T)$ and $(\lambda',T')$ respectively. Let $\pi(\O_N)=(S_i)_{i\in\bz_N}$ and $\pi'(\O_N)=(S_i')_{i\in\bz_N}$ be the two associated representations of the Cuntz algebra $\O_N$. Suppose that for every $i,j\in\bz_N$,
\begin{equation}
\sum_{n=1}^\infty|T_{i,j}^{(n)}-T_{i,j}'^{(n)}|<\infty.
\label{eq2.14.1}
\end{equation}
Then the measures $\mu$ and $\mu'$ are equivalent and the representations $\pi(\O_N)$ and $\pi'(\O_N)$ are also equivalent. 
\end{corollary}

\begin{proof}
By Proposition \ref{pr2.13}, we have
$$c:=\inf_{i,j}\inf_n T_{i,j}^{(n)}>0.$$
Let $x_1x_2\dots\in\K_N$. We have 
$$\sum_n\left|\frac{T_{x_n,x_{n+1}}'^{(n)}}{T_{x_n,x_{n+1}}^{(n)}}-1\right|\leq \frac1c\sum_n|{T_{x_n,x_{n+1}}'^{(n)}}-{T_{x_n,x_{n+1}}^{(n)}}|$$$$
=\frac1c\sum_{i,j\in\bz_N}\sum_{n: (x_n,x_{n+1})=(i,j)}|T_{i,j}^{(n)}-T_{i,j}^{(n)}|<\infty.$$

Therefore the infinite product
$$Z_\infty(x_1x_2\dots)=\frac{\lambda'_{x_1}}{\lambda_{x_1}}\prod_{n=1}^\infty\frac{T_{x_n,x_{n+1}}'^{(n)}}{T_{x_n,x_{n+1}}^{(n)}}$$
is convergent to a positive number. 
With Proposition \ref{pr2.12}, we obtain the results. 
\end{proof}

\begin{theorem}\label{th2.15}
Let $\mu$ be a quasi-stationary Markov measure associated to $(\lambda,T)$. Let $T^{(\infty)}$ be the limit of the stochastic matrices $T^{(n)}$ (see Proposition \ref{pr2.12}). Let $\lambda_\infty$ be a positive probability row vector such that $\lambda_\infty T^{(\infty)}=\lambda_\infty$. (The existence of such a vector follows from the Perron-Frobenius theorem).  Suppose the following condition is satisfied, for all $i,j\in\bz_N$:
\begin{equation}
\sum_{n=1}^\infty n|T_{i,j}^{(n)}-T_{i,j}^{(n+1)}|<\infty.
\label{eq2.15.1}
\end{equation}
Let $\mu_\infty$ be the Markov measure associated to $\lambda_\infty$ and the constant sequence with fixed matrix $T^{(\infty)}$. Then
\begin{enumerate}
	\item The measures $\mu$ and $\mu_\infty$ are equivalent.
	\item The monic representations of $\O_N$ associated to $\mu$ and to $\mu_\infty$ are equivalent.
\end{enumerate}
\end{theorem}

\begin{proof}
We have, for $i,j\in\bz_N$,
$$\sum_{n=1}^\infty|T_{i,j}^{(n)}-T_{i,j}^{(\infty)}|=\sum_{n=1}^\infty\left|T_{i,j}^{(n)}-\left(T_{i,j}^{(n)}+\sum_{k=n}^\infty(T_{i,j}^{(k+1)}-T_{i,j}^{(k)})\right)\right|\leq\sum_{n=1}^\infty\sum_{k=n}^\infty|T_{i,j}^{(k+1)}-T_{i,j}^{(k)}|$$
$$=\sum_{k=1}^\infty k|T_{i,j}^{(k+1)}-T_{i,j}^{(k)}|<\infty$$
With Corollary \ref{cor2.14} we get (i) and (ii). 
\end{proof}

\begin{acknowledgements}
This work was partially supported by a grant from the Simons Foundation (\#228539 to Dorin Dutkay). We thank Sergii Bezuglyi  for conversations about ergodic theory. One of the authors has had very helpful conversations with Professor Sergii Bezuglyi about Markov measures.
\end{acknowledgements}

\bibliographystyle{alpha}	
\bibliography{eframes}

\def\cprime{$'$} \def\cprime{$'$}
\begin{thebibliography}{KHL09}

\bibitem[AK08]{AK08}
Mitsuo Abe and Katsunori Kawamura.
\newblock Branching laws for endomorphisms of fermions and the {C}untz algebra
  {$\mathcal O_2$}.
\newblock {\em J. Math. Phys.}, 49(4):043501, 10, 2008.

\bibitem[Aki12]{Ak12}
Hasan Akin.
\newblock An upper bound of the directional entropy with respect to the
  {M}arkov measures.
\newblock {\em Internat. J. Bifur. Chaos Appl. Sci. Engrg.}, 22(11):1250263, 6,
  2012.

\bibitem[BJ02]{BrJo02}
Ola Bratteli and Palle E.~T. Jorgensen.
\newblock Wavelet filters and infinite-dimensional unitary groups.
\newblock In {\em Wavelet analysis and applications ({G}uangzhou, 1999)},
  volume~25 of {\em AMS/IP Stud. Adv. Math.}, pages 35--65. Amer. Math. Soc.,
  Providence, RI, 2002.

\bibitem[BR81]{BrRo81}
Ola Bratteli and Derek~W. Robinson.
\newblock {\em Operator algebras and quantum-statistical mechanics. {II}}.
\newblock Springer-Verlag, New York, 1981.
\newblock Equilibrium states. Models in quantum-statistical mechanics, Texts
  and Monographs in Physics.

\bibitem[Bur04]{Bur04}
Bernhard Burgstaller.
\newblock Slightly larger than a graph {$C^\ast$}-algebra.
\newblock {\em Israel J. Math.}, 144:1--14, 2004.

\bibitem[Cun77]{Cu77}
Joachim Cuntz.
\newblock Simple {$C\sp*$}-algebras generated by isometries.
\newblock {\em Comm. Math. Phys.}, 57(2):173--185, 1977.

\bibitem[Cun79]{Cu79}
J.~Cuntz.
\newblock Noncommutative {H}aar measure and algebraic finiteness conditions for
  simple {$C^{\ast} $}-algebras.
\newblock In {\em Alg\`ebres d'op\'erateurs et leurs applications en physique
  math\'ematique ({P}roc. {C}olloq., {M}arseille, 1977)}, volume 274 of {\em
  Colloq. Internat. CNRS}, pages 113--133. CNRS, Paris, 1979.

\bibitem[DHJ09]{DuJo9a}
Dorin~Ervin Dutkay, Deguang Han, and Palle E.~T. Jorgensen.
\newblock Orthogonal exponentials, translations, and {B}ohr completions.
\newblock {\em J. Funct. Anal.}, 257(9):2999--3019, 2009.

\bibitem[DHJ13]{DHJ13}
Dorin~Ervin Dutkay, John Haussermann, and Palle~E.T. Jorgensen.
\newblock Atomic representations of {C}untz algebras.
\newblock 2013.

\bibitem[DJ06]{DuJo06a}
Dorin~E. Dutkay and Palle E.~T. Jorgensen.
\newblock Wavelets on fractals.
\newblock {\em Rev. Mat. Iberoam.}, 22(1):131--180, 2006.

\bibitem[DJ07a]{DuJo07b}
Dorin~Ervin Dutkay and Palle E.~T. Jorgensen.
\newblock Analysis of orthogonality and of orbits in affine iterated function
  systems.
\newblock {\em Math. Z.}, 256(4):801--823, 2007.

\bibitem[DJ07b]{DuJo07a}
Dorin~Ervin Dutkay and Palle E.~T. Jorgensen.
\newblock Martingales, endomorphisms, and covariant systems of operators in
  {H}ilbert space.
\newblock {\em J. Operator Theory}, 58(2):269--310, 2007.

\bibitem[DJ08]{DuJo08a}
Dorin~Ervin Dutkay and Palle E.~T. Jorgensen.
\newblock Fourier series on fractals: a parallel with wavelet theory.
\newblock In {\em Radon transforms, geometry, and wavelets}, volume 464 of {\em
  Contemp. Math.}, pages 75--101. Amer. Math. Soc., Providence, RI, 2008.

\bibitem[DJ11]{DuJo11a}
Dorin~Ervin Dutkay and Palle E.~T. Jorgensen.
\newblock Affine fractals as boundaries and their harmonic analysis.
\newblock {\em Proc. Amer. Math. Soc.}, 139(9):3291--3305, 2011.

\bibitem[DJ12]{DuJo12}
Dorin~Ervin Dutkay and Palle E.~T. Jorgensen.
\newblock Spectral measures and {C}untz algebras.
\newblock {\em Math. Comp.}, 81(280):2275--2301, 2012.

\bibitem[DJ14]{DJ14}
Dorin~Ervin Dutkay and Palle~E.T. Jorgensen.
\newblock Monic representations of the {C}untz algebra and {M}arkov measures.
\newblock 2014.

\bibitem[DPS14]{DPS14}
Dorin~Ervin Dutkay, Gabriel Picioroaga, and Myung-Sin Song.
\newblock Orthonormal bases generated by {C}untz algebras.
\newblock {\em J. Math. Anal. Appl.}, 409(2):1128--1139, 2014.

\bibitem[ES80]{EnSh80}
H.~J. Engelbert and A.~N. Shiryaev.
\newblock On absolute continuity and singularity of probability measures.
\newblock In {\em Mathematical statistics}, volume~6 of {\em Banach Center
  Publ.}, pages 121--132. PWN, Warsaw, 1980.

\bibitem[Gli60]{Gli60}
James~G. Glimm.
\newblock On a certain class of operator algebras.
\newblock {\em Trans. Amer. Math. Soc.}, 95:318--340, 1960.

\bibitem[GN07]{GN07}
Rostislav Grigorchuk and Volodymyr Nekrashevych.
\newblock Self-similar groups, operator algebras and {S}chur complement.
\newblock {\em J. Mod. Dyn.}, 1(3):323--370, 2007.

\bibitem[GPS95]{GPS95}
Thierry Giordano, Ian~F. Putnam, and Christian~F. Skau.
\newblock Topological orbit equivalence and {$C^*$}-crossed products.
\newblock {\em J. Reine Angew. Math.}, 469:51--111, 1995.

\bibitem[Hid80]{Hi80}
Takeyuki Hida.
\newblock {\em Brownian motion}, volume~11 of {\em Applications of
  Mathematics}.
\newblock Springer-Verlag, New York, 1980.
\newblock Translated from the Japanese by the author and T. P. Speed.

\bibitem[Jor01]{Jor01}
Palle E.~T. Jorgensen.
\newblock Minimality of the data in wavelet filters.
\newblock {\em Adv. Math.}, 159(2):143--228, 2001.
\newblock With an appendix by Brian Treadway.

\bibitem[Jor06]{Jor06}
Palle E.~T. Jorgensen.
\newblock {\em Analysis and probability: wavelets, signals, fractals}, volume
  234 of {\em Graduate Texts in Mathematics}.
\newblock Springer, New York, 2006.

\bibitem[JP12]{JoPa12}
P.~E.~T. Jorgensen and A.~M. Paolucci.
\newblock Markov measures and extended zeta functions.
\newblock {\em J. Appl. Math. Comput.}, 38(1-2):305--323, 2012.

\bibitem[Kak48]{Ka48}
Shizuo Kakutani.
\newblock On equivalence of infinite product measures.
\newblock {\em Ann. of Math. (2)}, 49:214--224, 1948.

\bibitem[Kaw03]{Kaw03}
Katsunori Kawamura.
\newblock Generalized permutative representation of {C}untz algebra. {I}.
  {G}eneralization of cycle type.
\newblock {\em S\=urikaisekikenky\=usho K\=oky\=uroku}, (1300):1--23, 2003.
\newblock The structure of operator algebras and its applications (Japanese)
  (Kyoto, 2002).

\bibitem[Kaw06]{Kaw06}
Katsunori Kawamura.
\newblock Branching laws for polynomial endomorphisms of {C}untz algebras
  arising from permutations.
\newblock {\em Lett. Math. Phys.}, 77(2):111--126, 2006.

\bibitem[Kaw09]{Kaw09}
Katsunori Kawamura.
\newblock Universal fermionization of bosons on permutative representations of
  the {C}untz algebra {${\mathcal O}_2$}.
\newblock {\em J. Math. Phys.}, 50(5):053521, 9, 2009.

\bibitem[KHL09]{KHL09}
Katsunori Kawamura, Yoshiki Hayashi, and Dan Lascu.
\newblock Continued fraction expansions and permutative representations of the
  {C}untz algebra {${\mathcal O}_\infty$}.
\newblock {\em J. Number Theory}, 129(12):3069--3080, 2009.

\bibitem[Sen06]{Sen06}
E.~Seneta.
\newblock {\em Non-negative matrices and {M}arkov chains}.
\newblock Springer Series in Statistics. Springer, New York, 2006.
\newblock Revised reprint of the second (1981) edition [Springer-Verlag, New
  York; MR0719544].

\bibitem[Shi96]{Shi80}
A.~N. Shiryaev.
\newblock {\em Probability}, volume~95 of {\em Graduate Texts in Mathematics}.
\newblock Springer-Verlag, New York, second edition, 1996.
\newblock Translated from the first (1980) Russian edition by R. P. Boas.

\bibitem[Ska97]{Sk970}
Christian Skau.
\newblock Orbit structure of topological dynamical systems and its invariants.
\newblock In {\em Operator algebras and quantum field theory ({R}ome, 1996)},
  pages 533--544. Int. Press, Cambridge, MA, 1997.

\bibitem[Str89]{Str89}
Robert~S. Strichartz.
\newblock Besicovitch meets {W}iener-{F}ourier expansions and fractal measures.
\newblock {\em Bull. Amer. Math. Soc. (N.S.)}, 20(1):55--59, 1989.

\end{thebibliography}

\end{document}